\documentclass[11pt,letterpaper]{amsart}
\usepackage{amsfonts, amsmath, amssymb, amscd, amsthm, color, graphicx, mathrsfs, wasysym, setspace, mdwlist, calc, float, mathtools,enumitem}

\usepackage{xcolor}

\usepackage{hyperref} 

\hypersetup{colorlinks,
    linkcolor={red!50!black},
    citecolor={blue!80!black},
    urlcolor={blue!80!black}}
\usepackage{float}

\newcommand{\G}{\mathcal G}

\newcommand{\e}{\varepsilon}
\newcommand{\N}{\mathbb{N}}
\newcommand{\Z}{\mathbb{Z}}

\renewcommand{\d}{{\rm d}}
\renewcommand{\ll }{\langle\hspace{-.7mm}\langle }
\newcommand{\rr }{\rangle\hspace{-.7mm}\rangle }

\newcommand{\Cqi}{\sim_{C-q.i.}}

\newtheorem{thm}{Theorem}[section]
\newtheorem{cor}[thm]{Corollary}
\newtheorem{lem}[thm]{Lemma}
\newtheorem{prop}[thm]{Proposition}

\theoremstyle{definition}
\newtheorem{defn}[thm]{Definition}

\theoremstyle{remark}
\newtheorem{rem}[thm]{Remark}

\newtheorem{ex}[thm]{Example}

\renewcommand{\S}{\mathcal S}
\renewcommand{\L}{\mathcal L}
\newcommand{\Lw}{\mathcal L_{\omega_1, \omega}}
\newcommand{\DL}{\mathit{DL}}
\newfont{\eufm}{eufm10}

\begin{document}

\title{Quasi-isometric diversity of marked groups}

\author{A. Minasyan, D. Osin, S. Witzel}\thanks{The second named author was supported by the NSF grant DMS-1612473. The third named author was supported by a Feodor Lynen Research Fellowship of the Humboldt Foundation and the DFG Heisenberg grant WI 4079/6.}

\address[A. Minasyan]{School of Mathematical Sciences,
University of Southampton, Highfield, Southampton, SO17 1BJ, United
Kingdom.}
\email{aminasyan@gmail.com}
\address[D. Osin]{Department of Mathematics, Vanderbilt University,
1326 Stevenson Center, Nashville, TN 37240, USA.}
\email{denis.v.osin@vanderbilt.edu}
\address[S. Witzel]{Mathematical Institute, Gie\ss{}en University, Arndtstr. 2, 35392 Gie\ss{}en, Germany.}
\email{stefan.witzel@math.uni-giessen.de}
\date{}

\begin{abstract}
We use basic tools of descriptive set theory to prove that a closed set $\mathcal S$ of marked groups has $2^{\aleph_0}$ quasi-isometry classes provided every non-empty open subset of $\mathcal S$ contains at least two non-quasi-isometric groups. It follows that every perfect set of marked groups having a dense subset of finitely presented groups contains $2^{\aleph_0}$ quasi-isometry classes. These results account for most known constructions of continuous families of non-quasi-isometric finitely generated groups. We use them to prove the existence of $2^{\aleph_0}$ quasi-isometry classes of finitely generated groups having interesting algebraic, geometric, or model-theoretic properties (e.g., such groups can be torsion, simple, verbally complete or they
can all have the same elementary theory).
\end{abstract}

\maketitle

\section{Introduction}

Quasi-isometry is an equivalence relation that identifies metric spaces having the same large scale geometry. It is especially useful in geometric group theory and plays essentially the same role as the isomorphism relation in algebra.

It is well-known that there exist $2^{\aleph_0}$ quasi-isometry classes of finitely generated groups. Indeed, this immediately follows from the existence of $2^{\aleph_0}$ groups with pairwise inequivalent growth functions proved by Grigorchuk \cite{Gri}. A simpler argument using small cancellation theory was given by Bowditch in \cite{Bow}.

More recently, the question of quasi-isometric diversity of certain particular classes of groups has received considerable attention. For example, Cornulier and Tessera showed that there are $2^{\aleph_0}$ quasi-isometry classes of finitely generated solvable groups \cite{CT}; Brieussel and Zheng showed that such examples can be found among (locally finite)-by-cyclic groups. A continuous family of pairwise non-quasi-isometric groups with property $\mathit{FP}$ was constructed by Kropholler, Leary and Soroko \cite{KLS}. Gruber and Sisto showed that Gromov's construction of groups associated to a random labeling of certain expanders yields $2^{\aleph_0}$ quasi-isometry classes \cite{GS}.

Our main goal is to show that these results can be thought of as particular manifestations of a more general phenomenon. To state our main theorem we need to recall the definition of the space of marked groups introduced by Grigorchuk \cite{Gri}.

Let $\G_n$ denote the set of all pairs $(G,X)$, where $G$ is a group generated by an ordered $n$-tuple $X\in G^n$, considered up to the following equivalence relation: $(G, (x_1, \ldots, x_n))$ and $(H, (y_1, \ldots, y_n))$ are \emph{equivalent} if the map $x_i\to y_i$ extends to an isomorphism $G\to H$. Every such  pair can be naturally identified with a normal subgroup of $F_n$, the free group of rank $n$; namely, $(G,X)$ corresponds to the kernel of the natural homomorphism $F_n\to G$ induced by mapping a fixed basis of $F_n$ to $X$. The product topology on $2^{F_n}$ induces a structure of a second countable compact Hausdorff topological space on the set of normal subgroups of $F_n$ and, via the above identification, on $\mathcal G_n$. We  identify $\G_n$ with a clopen subspace of $\G_{n+1}$ via the map $$(G, (x_1, \ldots,x_n))\mapsto (G, (x_1, \ldots, x_n, 1)).$$  The topological union $\G=\bigcup_{n\in \mathbb N}\G_n$ is called the \emph{space of marked finitely generated groups.}

Recall that a subset of a topological space is said to be \emph{comeagre} if it is the intersection of countably many subsets with dense interiors. We will say that a subset $\mathcal{S}\subseteq \mathcal G$ is \emph{quasi-isometrically diverse} if  every comeagre subset of $\mathcal{S}$ has $2^{\aleph_0}$ quasi-isometry classes.

\begin{thm}\label{main}
Let $\mathcal{S}$ be a non-empty closed subset of $\mathcal G$. If every non-empty open subset of $\mathcal{S}$ contains at least two non-quasi-isometric groups, then $\mathcal{S}$ is quasi-isometrically diverse.
\end{thm}

The proof of Theorem \ref{main}, given in Section~\ref{sec:3}, is fairly short and combines an observation, originally made by Thomas \cite{T}, that $C$-quasi-isometry defines a closed relation on $\G\times\G$, with a purely topological argument. Nevertheless, Theorem~\ref{main} allows one to avoid non-trivial computations of particular invariants used to distinguish the quasi-isometry classes in \cite{Bow,CT,GS,KLS} via the following corollary. For a more detailed discussion, we refer to Examples \ref{ex1}--\ref{ex3}.

\begin{cor}\label{fp}
Every non-empty perfect subset of $\G$ containing a dense subset of finitely presented groups is quasi-isometrically diverse.
\end{cor}

It is worth noting that the assumption of the existence of a dense subset of finitely presented groups is essential in Corollary \ref{fp}. Indeed, there exist non-empty perfect subsets of $\mathcal G$ consisting of pairwise quasi-isometric groups, see Example \ref{pqi} below.

Corollary \ref{fp} can be used to prove quasi-isometric diversity of groups with certain interesting algebraic properties. Following the standard terminology, we say that a group-theoretic property $P$ is \emph{generic} in a subspace $\S$ of $\G$ if $P$ holds in a comeagre subset of $\S$. It is useful to keep in mind that any countable conjunction of generic properties is again generic since the class of comeagre sets in any topological space is closed under countable intersections. If $\S$ is quasi-isometrically diverse and $P$ is a generic property in $\S$, we obtain the existence of $2^{\aleph_0}$ pairwise non-quasi-isometric groups with property $P$.

One particular implementation of this idea, leading to a quasi-isometrically diverse zoo of `monsters', is considered in Section \ref{SecGen}. Let $\S=\overline{\mathcal H}_0$ be the closure of the set of all pairs $(G,X)\in \G$ such that $G$ is a non-elementary hyperbolic group without non-trivial finite normal subgroups. Generic properties in similar but slightly different classes of groups were studied by Champetier \cite{Cha}. Unfortunately, the paper \cite{Cha} does not cover some of the properties we are interested in (see the beginning of Section \ref{SecGen} for a more detailed discussion). The proposition below is a variation of Champetier's results. Similarly to \cite{Cha}, our proof is based on small cancellation theory for hyperbolic groups \cite{Gro,Ols}.

\begin{prop}\label{P1P3}
The properties of being torsion, simple, verbally complete, and being a common quotient of all non-elementary hyperbolic groups are generic in $\overline{\mathcal H}_0$.
\end{prop}

Recall that a group $G$ is \emph{verbally complete} if for every element $g \in G$ and each non-trivial reduced word $v(x_1,\dots,x_k)$, over an infinite alphabet $\{x_1, x_1^{-1}, x_2, x_2^{-1}, \ldots\}$, the equation $v(x_1,\dots,x_k)=g$  has a solution in $G$. In particular, every verbally complete group is \emph{divisible}; that is, the equation $x^n=g$ has a solution in $G$ for all $g\in G$ and $n\in \mathbb N$. Rational numbers provide the obvious example of a divisible group. For a long time, the existence of finitely generated divisible groups was an open question; the affirmative answer was given by Guba in \cite{Guba}.

Infinite common quotients of all non-elementary hyperbolic groups have property (T) of Kazhdan, and hence they are non-amenable (but not
uniformly non-amenable, see \cite{Osi}). In particular, generic groups in $\overline{\mathcal H}_0$ are examples of non-amenable groups without non-cyclic free subgroups. The existence of such groups was another famous open question, commonly referred to as the Day--von Neumann problem; it was solved by Olshanskii in \cite{Ols80}. Common quotients of all non-elementary hyperbolic groups also have strong fixed point properties for actions on finite-dimensional contractible Hausdorff spaces \cite{ABJLMS} as well as on $L^p$-spaces \cite{DM,NS}.

The set $\overline{\mathcal H}_0$ is known to be perfect (see, for example, \cite[Theorem 2.9]{Osi21}). Since every hyperbolic group is finitely presented, Corollary \ref{fp} can be applied to $\S=\overline{\mathcal H}_0$. Combining Corollary~\ref{fp} with
Proposition~\ref{P1P3}, we obtain the following.

\begin{cor}\label{h}
There exist $2^{\aleph_0}$  pairwise non-quasi-isometric  finitely generated groups which simultaneously satisfy all properties listed in Proposition \ref{P1P3}.
\end{cor}

In particular,  Corollary \ref{h} provides the first example of an uncountable family of quasi-isometry classes of finitely generated simple groups.
The same approach can be used to construct uncountable sets of pairwise non-quasi-isometric finitely generated groups with other interesting algebraic properties, e.g., torsion-free Tarski Monsters and groups with $2$ conjugacy classes. We briefly discuss these applications at the end of Section~\ref{SecGen}.

Our next result clarifies the relation between quasi-isometry and an equivalence relation studied in model theory. Recall that two groups are \emph{elementary equivalent} is they satisfy the same first-order sentences. In some classes of groups, elementary equivalence is closely related to quasi-isometry. For example, the former implies the latter for polycyclic groups \cite{Raph}.  We show that, in general, elementary equivalence does not impose any restrictions on quasi-isometric diversity. The proof is a simple combination of Corollary \ref{fp} and a recent result of the second author
\cite[Theorem~2.9]{Osi21} stating that $\overline{\mathcal H}_0$ contains a comeagre elementary equivalence class.

\begin{cor}\label{AE-Th}
There exist $2^{\aleph_0}$ finitely generated, pairwise non-quasi-isometric, elementarily equivalent groups.
\end{cor}

In the settings of Corollary \ref{fp},  generic groups in $\mathcal S$ will have `sparse' sets of relations. However, Theorem~\ref{main} can also be used to produce examples of non-quasi-isometric groups of completely different type, namely densely related groups in the sense of \cite{BC}. To illustrate this, we recover some of the results of Cornulier--Tessera \cite{CT} and Brieussel--Zheng \cite{BZ}.

\begin{cor}\label{cbm}
There exist $2^{\aleph_0}$ pairwise non-quasi-isometric finitely generated groups $G$ splitting as
$$
\{1\}\to {\Z_2}^\infty \to G\to \mathbb Z_2 {\,\rm wr\,}\mathbb Z \to \{1\},
$$
where ${\Z_2}^\infty$ is central in $G$. In particular,
\begin{enumerate}
\item[(a)] there exist continuously many quasi-isometry classes of finitely generated groups of asymptotic dimension $1$;
\item[(b)] there exist continuously many quasi-isometry classes of finitely generated center-by-metabelian groups (i.e., groups $G$ satisfying the identity $[G^{\prime\prime}, G]=1$).
\end{enumerate}
\end{cor}

Claim (a) also follows from the work of Brieussel and Zheng \cite{BZ} and should be compared with the well-known fact that every finitely presented group of asymptotic dimension $1$ is virtually free \cite{FW}. In particular, there are only $2$ quasi-isometry classes of finitely presented groups of asymptotic dimension $1$: virtually infinite cyclic and virtually non-cyclic free groups. Claim (b) is a slight improvement of the result of Cornulier and Tessera \cite{CT}, who constructed continuously many quasi-isometry classes of finitely generated solvable groups; in fact, their groups are (nilpotent of class $3$)-by-abelian, while our groups are (nilpotent of class $2$)-by-abelian. Note that there is no hope to prove the same result for solvable groups of derived length $2$ as finitely generated metabelian groups satisfy the maximum condition for normal subgroups \cite{Hall} and, therefore, the set of such groups is countable.

\section{Preliminaries}
Let $G$ be a group, and let $X=(x_1,\dots,x_n) \in G^n$ be an ordered $n$-tuple of elements of $G$, for some $n \in \N$.
We will say that $X$ is a generating set of $G$ if $x_1,\dots,x_n$ generate $G$ in the usual sense. If $H$ is another group and
$\phi\colon G \to H$ is an epimorphism, the generating tuple (set) $X=(x_1,\dots,x_n)$ of $G$ gives rise to a generating tuple (set) $\phi(X)=(\phi(x_1),\dots,\phi(x_n))$ of $H$.

Given a group $G$ generated by a set $X$, we denote by $|\cdot |_X $ the word length with respect to $X$ and by $\d_X$ the corresponding word metric on $G$; by $\Gamma (G,X)$ we denote the Cayley graph of $G$ with respect to $X$. The definition of the space of finitely generated marked groups given below is slightly different from (but obviously equivalent to) the one given in the Introduction.

For a group $G$, we denote by $\mathcal N(G)$ the set of all normal subgroups of $G$ endowed with the topology induced by the product topology on $2^{G}$. Recall that the product topology on $2^G$ is simply the topology of pointwise convergence of indicator functions. So, an open neighborhood of $N \in \mathcal{N}(G)$ consists of all subgroups $M \lhd G$ such that $N \cap P=M \cap P$, for some finite subset $P \subseteq G$.

Given any $n\in \mathbb N$, let $F_n$ denote the free group with a fixed basis $\{ a_1,\ldots, a_n\}$. Every pair $(G,X)$, where $G$ is a group generated by an ordered tuple $X=(x_1,\dots,x_n)$, corresponds to a normal subgroup of $F_n$, namely the kernel of the homomorphism $\e_{G,X}\colon F_n\to G$ sending $a_i$ to $x_i$ for all $i$. We say that two such pairs $(G,X)$ and $(H, Y)$, where $|X|=|Y|=n$, are \emph{equivalent} if $Ker(\e_{G,X}) = Ker(\e_{H,Y})$.
\begin{ex}
Let $n=2$ and let $G\cong \mathbb Z\oplus\mathbb Z$. Then for any generating pairs $X=(x_1,x_2)$ and $Y=(y_1,y_2)$ of $G$, the pairs $(G,X)$ and $(G,Y)$ are equivalent since $\e_{G,X}=\e_{G,Y}=[F_2, F_2]$.
\end{ex}

For each $n \in \N$ we denote by $\mathcal G_n$ the set of equivalence classes of pairs $(G,X)$, where $G$ is a group and $X$ is a generating $n$-tuple of $G$, and endow it with the pull-back of the product topology on $\mathcal N (F_n)$ via the (bijective) map $(G,X)\mapsto Ker(\e_{G,X})$. For simplicity, we keep the notation $(G,X)$ for the equivalence class of $(G,X)$. The map  $(G,(x_1, \ldots, x_n))\mapsto (G,(x_1, \ldots, x_n, 1))$ yields a topological embedding $\G_n \to \G_{n+1}$. Note that under this embedding the image of $\G_n$ is both closed and open in $\G_{n+1}$. 
The topological union  $\G=\bigcup_{n\in \mathbb N}\G_n$ is called the \emph{space of finitely generated marked groups}. Thus a subset $U$ of $\G$ is open  if and only if 
$U =\bigcup_{n \in \N} U_n$, where $U_n$ is an open subset of $\G_n$ (here we identify $\G_n$ with its image in $\G$). It follows that a sequence $(x_k)_{k \in \N}$ 
converges to $x \in \G_n \subset\G$ if and only if there is $K \in \N$ such that $x_k \in \G_n$ for all $k \ge K$ and the tail sequence $(x_k)_{k \ge K}$ converges to $x$ in $\G_n$.

The topology on $\G$ can be visualized in terms of Cayley graphs as follows. For a group $G$, generated by a set $X$, and for any $r>0$, we define the \emph{ball of radius $r$ centered at the identity element} by
$$
B_{G,X}(r)=\{ g\in G \mid |g|_X\le r\}.
$$

Let $(G,X)\in \G_m$ and $(H,Y)\in \G_n$, where $X=(x_1, \ldots, x_m)$ and $Y=(y_1, \ldots, y_n)$. We will say that  $(G,X)$ and $(H,Y)$ are \emph{$r$-locally isomorphic}, for some $r\in \mathbb N$, if $m=n$ and there exists an isomorphism between the induced subgraphs of the Cayley graphs $\Gamma (G,X)$ and $\Gamma (H,Y)$ with the vertex sets $B_{G,X}(r)$ and $B_{H,Y}(r)$, respectively, that maps the identity element of $G$ to the identity element of $H$ and edges labelled by $x_i$ to edges labelled by $y_i$, for all $i=1, \ldots, n$. It is easy to see that $r$-local isomorphism is a well-defined equivalence relation on $\G$. A sequence of marked  groups $(G_i, X_i)$ converges to $(G,X)$ in $\mathcal G$ if for every $r\in \mathbb N$,  $(G_i, X_i)$ and $(G,X)$ are $r$-locally isomorphic for all large enough $i \in \N$.

We now record a useful observation.

\begin{lem}\label{NG}
Let $G$ be a group generated by a finite set $X$. Given a normal subgroup $N\lhd G$, we let $X_N$ denote the image of $X$ under the natural homomorphism $G\to G/N$. The map $\mathcal N(G)\to \G$ defined by the formula $N\mapsto (G/N, X_N)$ for all $N\in \mathcal N(G)$ is continuous.
\end{lem}

\begin{proof}
Suppose that a sequence $(N_i)$ converges to $N$ in $\mathcal N(G)$. Fix any $r\in \mathbb N$. For all sufficiently large $i$, we have $N_i\cap B_{G,X}(2r)= N\cap B_{G,X}(2r)$. For every such $i$, $(G/{N_i}, X_{N_i})$ is $r$-locally isomorphic to $(G/{N}, X_{N})$ and the claim follows.
\end{proof}

\begin{ex}\label{limex}
Suppose that we have a sequence of groups and epimorphisms
\begin{equation}\label{sec}
G\stackrel{\e_0}\longrightarrow G_1 \stackrel{\e_1}\longrightarrow G_2 \stackrel{\e_2}\longrightarrow\ldots .
\end{equation}
Let $X$ be a generating set of $G$, $X_i=\e_{i-1}\circ \cdots \circ \e_0(X)$ and $N_i=\ker(\e_{i-1}\circ \cdots \circ \e_0)$, $i \in \N$.
Observe that $(N_i)$ is a nested sequence of normal subgroups of $G$ converging to $N=\bigcup_{i\in \mathbb N} N_i$ in $\mathcal N(G)$. Let $G_\infty =G/N$ be the direct limit of the sequence
\eqref{sec}, and let $X_\infty$ denote the natural image of $X$ in $G_\infty$. Then $(G_i, X_i)\to (G_\infty, X_\infty)$ in $\mathcal G$ by Lemma~\ref{NG}.
\end{ex}

Let $(S, \d_S)$ and $(T, \d_T)$ be metric spaces. A map $\varphi \colon S\to T$ is called \emph{$C$-coarsely Lipschitz} for some constant $C$ if \begin{equation}\label{C-Lip}
\d_T(\varphi(x), \varphi(y)) \le C \d_S (x,y)+ C \mbox{ for all } x,y\in S.
\end{equation}
Two metric spaces $(S, \d_S)$ and $(T, \d_T)$ are said to be \emph{quasi-isometric} if there exists a constant $C$ and $C$-coarsely Lipschitz maps $\varphi\colon S\to T$ and $\psi\colon T\to S$ such that $\varphi$ and $\psi$ are \emph{$C$-coarse inverses} of each other; that is,
\begin{equation}\label{ci}
\sup\limits_{t\in T} \d_T (\varphi\circ \psi(t), t)\le C \;\;\; {\rm and }\;\;\; \sup\limits_{s\in S} \d_S (\psi\circ \varphi(s), s)\le C.
\end{equation}
The quasi-isometry relation between metric spaces is denoted by $\sim_{q.i.}$.

We now restrict to the case when $(S, \d_S)=(G, \d_X)$ and $(T, \d_T)=(H, \d_Y)$ are groups endowed with word metrics with respect to some generating sets $X$ and $Y$. Two marked groups $(G, X)$ and $(H, Y)$ are \emph{quasi-isometric} if so are the metric spaces $(G, \d_X)$ and $(H, \d_Y)$. Further, we say that marked groups $(G, X)$ and $(H, Y)$ are \emph{$C$-quasi-isometric} for some constant $C>0$ and write  $(G, X)\Cqi (H, Y)$ if there exist maps $\varphi$, $\psi$ as above satisfying (\ref{C-Lip}), (\ref{ci}), and the additional property
\begin{equation}\label{pt}
\varphi(1)=1 \;\;\; {\rm and}\;\;\; \psi(1)=1;
\end{equation}
in what follows, we refer to $\varphi$, $\psi$ as a \emph{$C$-quasi-isometry witnessing pair} for $(G, X)$ and $(H, Y)$. Note that, unlike $\sim_{q.i.}$, the relation $\Cqi$ is not transitive.

If $(G,X)\sim_{q.i.} (H,Y)$, we can always adjust the corresponding maps $\varphi$ and $\psi$ satisfying (\ref{C-Lip}) and (\ref{ci}) by redefining them at the identity elements so that they enjoy the additional property (\ref{pt}) and are still $C^\prime$-coarsely Lipschitz and $C^\prime$-coarsely mutually inverse for some (possibly larger) constant $C^\prime$. Thus, for the binary relations $\sim_{q.i.}$ and $\Cqi$ on $\mathcal G$, considered as subsets of $\mathcal G\times \mathcal G$, we have
\begin{equation}\label{u}
\sim _{q.i.}=\bigcup\limits_{C\in \mathbb N} \Cqi.
\end{equation}

We will need the following lemma, which can be extracted from the proof of Theorem 3.2 in \cite{T}. Since this result is crucial for our paper, we provide a brief proof for convenience of the reader. Our proof is slightly different from the one suggested in \cite{T}.

\begin{lem}[Thomas]\label{Thomas}
For any $C\in \mathbb N$, the relation $\Cqi$ is closed in $\mathcal G \times \mathcal G$.
\end{lem}

\begin{proof}
Let $(G,X)$ and $(H,Y)$ be two marked groups such that $((G,X), (H,Y))$ is contained in the closure of $\Cqi$ in $\mathcal G \times \mathcal G$. Approximating $(G,X)$ and $(H,Y)$ by $C$-quasi-isometric marked groups, for every $M\in \mathbb N$, we can find maps $\varphi_M\colon B_{G,X} (M)\to H$ and $\psi_M \colon B_{H,Y} (M)\to G$ such that

\begin{enumerate}[itemsep=10pt] 
\item[(a)] $\varphi_M(1)=1$ and $\psi_M(1)=1$;
\item[(b)] for every $g_1, g_2\in B_{G,X}(M)$, we have $\d_Y(\varphi_M(g_1), \varphi_M(g_2))\le C\d_X(g_1, g_2) +C $; similarly for $\psi_M$;
\item[(c)] we have $\d_X(\psi_M\circ \varphi_M (g), g)\le C$ whenever $\psi_M\circ \varphi_M (g)$  is well-defined; similarly for  $\varphi_M\circ \psi_M$.
\end{enumerate}

By (a) and (b), we have $\varphi _M (g)\in B_{H,Y}(C|g|_X+C)$ for all $g\in G$ and $M\ge |g|_X$. Thus, for every fixed $g\in G$, there are only finitely many possibilities for $\varphi_M (g)$ over all $M\ge |g|_X$. It follows that there exists a subsequence $( \varphi_{M_i})_{i\in \mathbb N}$ which converges pointwise to a map $\varphi\colon G\to H$; that is,
for each $g \in G$ we have $\varphi(g)=\varphi_{M_i}(g)$ for all but finitely many indices $i\in \mathbb N$. By the same argument applied to the maps $\psi_{M_i}$, we can assume, without loss of generality, that the sequence $( \psi_{M_i})_{i\in \mathbb N}$ also converges pointwise to some map $\psi\colon H\to G$.
Obviously $\varphi $ and $\psi$ satisfy all necessary conditions to be a $C$-quasi-isometry witnessing pair for $(G, X)$ and $(H, Y)$.
\end{proof}

\section{Proof of the main theorem}\label{sec:3}
Note that for each $n \in \N$, $\G_n$ is a compact metrizable open subspace of $\G=\bigcup_{n \in \N} \G_n$. Therefore the space $\G$ is locally compact and second countable, hence it is 
a Polish space (cf. \cite[Theorem I.5.3]{Kech}).

Recall that a subset of a topological space is \emph{perfect} if  it is closed and has no isolated points. It is well-known that every non-empty perfect subset of a Polish space has cardinality $2^{\aleph_0}$ (see \cite[Corollary~I.6.3]{Kech}). To prove our main theorem we will also need a classical result of Mycielski \cite{Myc}.

\begin{thm}[Mycielski]
Let $\mathcal E$ be an equivalence relation on a non-empty Polish space $\mathcal{S}$. If $\mathcal E$ is a meagre subset of $\mathcal{S}\times \mathcal{S}$, then there exists a non-empty perfect subset $\mathcal{P}\subseteq \mathcal{S}$ such that $(x,y)\notin \mathcal E$ for any two distinct $x,y\in \mathcal{P}$.
\end{thm}

\begin{proof}[Proof of Theorem \ref{main}]
Lemma \ref{Thomas} and the assumption that every non-empty open subset of $\mathcal{S}$ contains non-quasi-isometric groups imply that $\Cqi$ is a nowhere dense subset of $\mathcal{S}\times \mathcal{S}$ for every $C>0$. Since a countable union of nowhere dense subsets is meagre and $\sim_{q.i.}$ is the union of $\Cqi$ for all $C\in \mathbb N$, see (\ref{u}), the relation $\sim_{q.i}$ is meagre in $\mathcal{S}\times \mathcal{S}$.

Note that $\mathcal{S}$ is a Polish space being a closed subset of the Polish space $\mathcal G$.
Let $\mathcal{D} \subseteq \mathcal{S}$ be a comeagre subset. By the Baire category theorem, $\mathcal{D}$ contains a dense $G_\delta$-subset $\mathcal{D}_0$ of $\mathcal{S}$. Note that $\mathcal{D}_0$ is itself a Polish space being a $G_\delta$-subset of the Polish space $\mathcal{S}$ (see \cite[Theorem~I.3.11]{Kech}).
Since $\sim_{q.i}$ is meagre in $\mathcal{S}\times \mathcal{S}$, it follows that $\sim_{q.i}$ is also meagre in $\mathcal{D}_0 \times \mathcal{D}_0$
(indeed, it is easy to see that for a meagre subset $\mathcal{M}$ and a dense subset $\mathcal{Y}$ in a topological space, $\mathcal{M} \cap \mathcal{Y}$ is meagre in $\mathcal{Y}$).
Therefore, by Mycielski's theorem, there exists a non-empty perfect subset $\mathcal{P}\subseteq \mathcal{D}_0 \subseteq \mathcal{D}$ consisting of pairwise non-quasi-isometric groups. In particular, $\mathcal P$ has cardinality $2^{\aleph_0}$.
\end{proof}

We now prove Corollary \ref{fp} and discuss its relation to some previously known results.

\begin{proof}[Proof of Corollary \ref{fp}]
Let $\mathcal{S}$ be a perfect subset of $\mathcal G$ that contains a dense subset of finitely presented groups. Let $\mathcal{U}$ be any non-empty open subset of $\mathcal{S}$. By our assumption, there is a finitely presented marked group $(G,X)\in \mathcal{U}$. Being a non-empty open subset of a completely metrizable space without isolated points, $\mathcal{U}$ is uncountable
by the Baire category theorem. Since there are only countably many finitely presented groups, $\mathcal{U}$ contains a marked group $(H,Y)$ that is not finitely presented. Recall that  being finitely presented is a quasi-isometry invariant
(see \cite[I.8.24]{BH}), therefore, $G\not\sim_{q.i.} H$ and Theorem \ref{main} applies.
\end{proof}

\begin{ex}\label{ex1}
In \cite{Bow} Bowditch considers an infinite set of words $w_1, w_2, \ldots $ in the alphabet $\{a, b\}$, satisfying the $C^\prime (1/6)$ small cancellation condition, and for each $I\subseteq \mathbb N$ defines the group \[G_I=\langle a,b\mid w_i=1, \, i\in I\rangle .\]
He then uses a quasi-isometric invariant based of the notion of a \emph{taut loop} to distinguish uncountably many quasi-isometry classes in the set $ \mathcal{S}=\{ G_I\mid I\subseteq \mathbb N\}$.

We note that computing this invariant is unnecessary. Indeed Greendlinger's lemma for $C^\prime (1/6)$ presentations \cite[V.4.5]{LS} ensures that the map $2^\mathbb N\to \mathcal G$, defined by $I\mapsto (G_I,\{a,b\})$, is injective and continuous. In particular, its image is homeomorphic to the Cantor set and contains a dense subset of finitely presented groups (corresponding to finite subsets of $\mathbb N$). Therefore, Corollary \ref{fp} applies.
\end{ex}

\begin{ex}\label{ex2}
In \cite{CT}, Cornulier and Tessera start with a finitely presented solvable group $G$, which has a central subgroup of the form $E=\oplus_{i\in \mathbb N} E_i$, where $E_i\cong \mathbb Z_p$ for all $i$ and some fixed prime $p$. Then they use a rather sophisticated invariant -- the set of ultrafilters, for which the asymptotic cones of a group are simply connected -- to distinguish uncountably many non-quasi-isometric quotients of $G$. Our approach allows one to avoid computing this invariant as every such a group $G$ always has $2^{\aleph_0}$ non-quasi-isometric quotients. This result is an easy consequence of Corollary \ref{fp}. For an expositional reason, we postpone the proof till Section~\ref{sec:central_ext} (see Corollary \ref{fpZ}) as it uses the notation introduced there.
\end{ex}

\begin{ex}\label{ex3} We leave it to the reader to verify that the proof of quasi-isometric diversity of groups constructed by Kropholler--Leary--Soroko \cite{KLS} and Gruber--Sisto \cite{GS} can also be simplified to Corollary \ref{fp}. More precisely, our result allows one to avoid computing Bowditch's invariant in \cite{KLS} and divergence functions in \cite{GS} for the purpose of distinguishing non-quasi-isometric groups. Note, however, that we do need to use some technical lemmas from these papers to establish that the corresponding subspaces of $\G$ are perfect.
\end{ex}

\begin{ex}\label{ex4}
The quasi-isometric diversity of groups constructed by Grigorchuk \cite{Gri} cannot be derived from Corollary \ref{fp}. However, it can be explained using Theorem \ref{main}. Indeed, Grigorchuk shows that there exists a (non-empty) perfect subset of $\mathcal G$, which contains a dense subset of groups of exponential growth and a dense subset of groups of subexponetial growth; therefore, Theorem \ref{main} applies. Granted, proving the facts mentioned in the previous sentence is not much easier than distinguishing $2^{\aleph_0}$ inequivalent growth functions. Thus, unlike in Examples \ref{ex1}--\ref{ex3}, our approach does not help much here (and, after all, constructing non-quasi-isometric groups was not the main purpose of \cite{Gri}).
\end{ex}

\section{Generic properties and limits of hyperbolic groups}\label{SecGen}

Every hyperbolic group $G$ has a maximal finite normal subgroup, which we denote by $E(G)$ and call the \emph{finite radical} of $G$. A hyperbolic group is called \emph{non-elementary} if it is not virtually cyclic. We denote by $\mathcal H_0$ the set of all marked groups $(G,X)$ such that $G$ is non-elementary, hyperbolic, and has trivial finite radical.

The main goal of this section is to prove Proposition \ref{P1P3}. Formally, by a \emph{group-theoretic property} we mean a predicate $$P\colon \G \to \{ true,\; false\}$$ that is constant on isomorphism classes, i.e., we have $P(G,X)=P(H,Y)$ whenever $G\cong H$. We also say that a finitely generated group $G$ \emph{has property $P$} if $P(G,X)=true$ for some (equivalently, any) finite generating set $X$.

For any subset $\mathcal S\subseteq \mathcal G$ we will write $\overline{\mathcal S}$ to denote the closure of $\mathcal S$ in $\G$. If $P$ is any group-theoretic property, we define
\[
P(\mathcal S) =\{ (G,X)\in \mathcal S \mid P(G,X)=true \}.
\]

Recall that a subset of a topological space is called a \emph{$G_\delta$-set} if it can be represented as a countable intersection of open sets. By the Baire category theorem, a subset of a Polish space is comeagre if and only if it contains a dense $G_\delta$-set.

\begin{defn}
A group-theoretic property $P$ is a \emph{$G_\delta$-property} if $P(\G)$ is a $G_\delta$-set. Further, we say that $P$ is \emph{generic} in some $\mathcal S\subseteq \G$ if $P(\S)$ is comeagre in $\mathcal S$.
\end{defn}

The idea of studying generic properties of marked groups goes back to the paper \cite{Cha} by Champetier, who proved that groups with certain exotic properties (e.g., non-amenable groups without free subgroups) are generic in the closure of the set $\mathcal H_{cc}$, of all pairs $(G,X)\in \G$ such that $G$ is non-elementary, hyperbolic, and the centralizer of every element in $G\setminus\{1\}$ is cyclic. Note that $\mathcal H_{cc}\subseteq \mathcal H_0$, but these classes are not equal. E.g., the group $(\mathbb Z/2)^2 \ast \mathbb Z$ belongs to $\mathcal H_0$ but not to $\mathcal H_{cc}$. In the same paper, Champetier also proved that torsion groups are generic in the closure of $\mathcal H$, where $\mathcal H$ is the set of all non-elementary hyperbolic marked groups.

In this paper, we prefer to work with $\overline{\mathcal H}_0$ because generic groups in this set have (arguably) more interesting properties than those in $\overline{\mathcal H}$ or $\overline{\mathcal H}_{cc}$. For example, it is not difficult to show that generic groups in $\overline{\mathcal H}$ are not simple. Generic groups in $\overline{\mathcal H}_{cc}$ are, in fact, simple (although this was not proved in \cite{Cha}). However, no group in $\overline{\mathcal H}_{cc}$ is a common quotient of all non-elementary hyperbolic groups. It is also not clear if generic groups in $\overline{\mathcal H}_{cc}$ are torsion. In contrast, generic groups in $\overline{\mathcal H}_{0}$ enjoy all these properties.

Our proof of Proposition \ref{P1P3} has two ingredients.  The first one makes use of basic model theory. We begin by reviewing the necessary background.

Let $\L$ denote the language of groups. That is, $\L$ is the first-order language with the signature $\{ 1, \cdot, ^{-1}\}$. Recall that $\L_{\omega_1, \omega}$ is the infinitary version of $\L$, where terms, atomic formulas, and formulas are constructed from the symbols of $\mathcal L$ using the same syntactic rules as in the standard first-order logic with one exception: countably infinite conjunctions and disjunctions of formulas are allowed. For the formal definition, see \cite[Chapter 1]{Mar16}. As in the first-order logic, an $\L_{\omega_1, \omega}$-\emph{sentence} is an $\L_{\omega_1, \omega}$-formula without free variables.

\begin{ex}\label{tor}
Consider the $\Lw$-sentence
$$
\tau \colon \;\;\; \forall \, g \; \bigvee_{n\in \N} g^n=1
$$
Clearly, $G\models \tau$ if and only if $G$ is torsion.
\end{ex}

We will need a particular class of $\Lw$-formulas, denoted by $\Pi_2$, which is part of a more general hierarchy. We say that a class of formulas $\Theta$ is \emph{closed under adding universal (respectively, existential) quantifiers} if for every formula $\theta (x)\in \Theta$ with a free variable $x$, we have $\forall\, x\; \theta(x)\in \Theta$ (respectively, $\exists\, x\; \theta(x)\in \Theta$).

\begin{defn}
Let $\Sigma_1$ denote the smallest set of all $\Lw$-formulas that contains all finite Boolean combinations of atomic formulas and is closed under countable disjunctions, finite conjunctions, and adding existential quantifiers. By $\Pi_2$ we denote the set of all $\Lw$-formulas that contains all formulas from $\Sigma_1$ and is closed under countable conjunctions, finite disjunctions, and adding universal quantifiers. A group-theoretic property $P$ is \emph{$\Pi_2$-definable} if there exists an $\Pi_2$-sentence $\sigma$ such that $P(\G) = \{ (G,X)\mid G\models \sigma\}$.
\end{defn}

For instance, the formula $\tau$ in Example \ref{tor} is in the class $\Pi_2$ (but not in $\Sigma _1$). Therefore, being torsion is a $\Pi_2$-definable property.

The relevance of the infinitary logic to the study of generic properties is demonstrated by the proposition below, which has long been known to experts. The proof is fairly elementary and can be found in \cite[Proposition 5.1]{Osi21} (see also \cite[Proposition 3.2]{KOO}), but we provide a sketch for convenience of the reader.

\begin{prop}[folklore]\label{Gd}
A $\Pi_2$-definable group-theoretic property is a $G_\delta$-property.
\end{prop}

\begin{proof}
Consider the language $\mathcal L^\prime$ obtained from $\mathcal L$ by adding constants $c_1, c_2, \ldots$. Given $(G,X)\in \G$, where $X=(x_1, \ldots, x_n)$, we interpret $c_i$ as $x_i$ for $1\le i\le n$ and as $1$ for $i>n$. Thus every $(G,X)\in \G$ becomes an $\mathcal L^\prime$-structure. We also denote by $\mathcal L^\prime _{\omega_1, \omega}$ the corresponding infinitary version of $\mathcal L^\prime$.

Let $T$ denote the set of all words in the alphabet $\{c_1, c_1^{-1}, c_2, c_2^{-1},\ldots \}$. For every $\mathcal L^\prime_{\omega_1, \omega}$-formula $\phi (x)$ with one free variable $x$ and for every $(G,X)\in \G$, we obviously have
\begin{equation}\label{eq1}
(G,X) \models \forall x\, \phi(x) \;\; {\rm if\; and\; only\; if} \;\; (G,X) \models \bigwedge_{t\in T} \phi (t)
\end{equation}
and
\begin{equation}\label{eq2}
(G,X) \models \exists x\, \phi(x) \;\; {\rm if\; and\; only\; if}\;\; (G,X) \models \bigvee_{t\in T} \phi (t).
\end{equation}
Using (\ref{eq1}) and (\ref{eq2}) we can eliminate quantifiers in $\L^\prime$-formulas and show that every $\Pi_2$-sentence $\sigma$ in  $\L_{\omega_1,\omega}$ is equivalent to a countable conjunction of countable disjunctions of Boolean combinations of atomic formulas of the form $t_1=t_2$ for some $t_1, t_2\in T$. It is easy to verify that the latter formulas define clopen subsets of $\G$ and hence the set of models of $\sigma$ is a countable intersection of open sets.
\end{proof}

For brevity, we introduce the following notation for group-theoretic properties under consideration.
\begin{enumerate}
\item[$Tor$:] being torsion;
\item[$S$:] being simple;
\item[$VC$:] being verbally complete;
\item[$QHyp$:] being a common quotient of all non-elementary hyperbolic groups.
\end{enumerate}

\begin{cor}\label{Gdcor}
The properties $Tor$, $S$, $VC$, and $QHyp$ are $G_\delta$-properties.
\end{cor}

\begin{proof}
For property $Tor$, this result is due to Champetier \cite{Cha}. It also follows immediately from Proposition \ref{Gd} and Example \ref{tor} since $\tau \in \Pi_2$.

For property $S$, this result is proved in \cite[Proposition 3.4]{KOO}. It is also easy to see that $S$ can be expressed by the $\Pi_2$-formula
\[
\forall \, g \; \forall\, h\; \left( h=1\vee \left(\bigvee_{n\in \N} \exists\,t_1\; \ldots \exists\, t_n  \bigvee_{(\e_1, \ldots, \e_n)\in \{ -1, 1\}^n} g=t_1^{-1}h^{\e_1}t_1\cdots t_n^{-1}h^{\e_n}t_n\right) \right).
\]

Further, for every word $v=v(a_1, \ldots, a_k)$ in the infinite alphabet $$\mathcal A= \{a_1, a_1^{-1}, a_2, a_2^{-1}, \ldots\},$$ we consider the formula
$$
\varkappa _v\colon \;\;\; \forall\, g\; \exists\, a_1 \, \ldots\,  \exists\, a_k \;\; v(a_1, \ldots, a_k)=g.
$$
It is easy to see that being verbally complete means exactly that $G\models \varkappa_v$ for all words $v$ which are not equal to $1$ in the free group with basis $\{ a_1,a_2, \ldots\}$. Obviously, $\varkappa_v \in \Pi_2$ for any $v$. Since a countable conjunction of $\Pi_2$-sentences is again a $\Pi_2$-sentence, property $VC$ is $\Pi_2$-definable.

Finally, it is well-known and easy to show that for any finitely presented group $H$, the set of quotients of $H$ in $\G$ is open. Since hyperbolic groups are finitely presented and there are only countably many of them, $QHyp$ is a $G_\delta$-property.
\end{proof}

The second ingredient in the proof of Proposition \ref{P1P3} comes from small cancellation theory in hyperbolic groups, which was developed by  Olshanskii \cite{Ols} following ideas of Gromov \cite{Gro} and the earlier work on the geometric solution of the Burnside problem \cite{Ols82}.  We will need the following particular case of \cite[Theorem~2]{Ols}.

\begin{thm}[Olshanskii]\label{thm:Ols-1} Let $G$ be a non-elementary hyperbolic group with $E(G)=\{1\}$,
and let $H_1, \ldots, H_k$ be subgroups of $G$ such that each $H_i$ is not virtually cyclic and does not normalize any non-trivial finite subgroup of $G$. Then for every finite subset $M \subset G$ there exists a group $Q$ and an epimorphism $\varphi\colon G \to Q$ such that
\begin{enumerate}
  \item[(a)] $Q$ is non-elementary hyperbolic and $E(Q)=\{ 1\}$;
  \item[(b)] the restriction of $\varphi$ to $M$ is injective;
  \item[(c)] $\varphi (H_i)=Q$ for all $i=1, \ldots, k$;
  \item[(d)] if $G$ is torsion-free then so is $Q$;
  \item[(e)] for every $a\in M$, $C_Q(\varphi(a))=\varphi(C_G(a))$.
\end{enumerate}
\end{thm}

\begin{rem} The second part of condition (a) was not stated in \cite{Ols}, but can be proved using the same technique; see, for example, claim 9 of \cite[Theorem 1]{Min}.
\end{rem}

By abuse of notation, we write $G\in \S$ for a finitely generated group $G$ and a subset $\S\subseteq \G$, if $(G,X)\in \S$ for some finite generating set $X$ of $G$. For an element $g \in G$, we denote by $\ll g \rr$ the normal closure of $g$ in the group $G$.

\begin{lem}\label{nc} For any $G\in \mathcal H_0$ and any finite subsets $M \subset G$ and $L\subset G\setminus\{1\}$,
there exists $Q\in \mathcal H_0$ and an epimorphism $\varphi\colon G \to Q$ such that the restriction of $\varphi$ to $M$ is injective
and $Q= \ll\varphi(g) \rr$ for every $g \in L$.
\end{lem}

\begin{proof}
Let $L=\{g_1,\dots,g_k\}$ be the list of all elements of $L$, and set $H_i=\ll g_i\rr$, $i=1,\dots,k$. Since $E(G)=\{1\}$ and $g_i \neq 1$,
$H_i$ is an infinite normal subgroup of $G$, for each $i=1,\dots,k$. A virtually cyclic subgroup of a hyperbolic group is necessarily quasiconvex (see \cite[III.$\Gamma$.3.10 and III.$\Gamma$.3.6]{BH}). Hence every infinite virtually cyclic subgroup of $G$ must have finite index in its normalizer (see \cite[III.$\Gamma$.3.16]{BH}). As $G$ is non-elementary, it follows that $H_i$ is not virtually cyclic, for each $i =1,\dots,k$.

By \cite[Proposition 1]{Ols}, for each $i$, there exists a maximal finite subgroup $E_i \leqslant G$, normalized by $H_i$. Since $H_i \lhd G$, $g^{-1}E_ig$ is also normalized by $H_i$ for every $g\in G$ and, therefore, $g^{-1}E_ig=E_i$, by maximality. It follows that $E_i\lhd G$,  hence $E_i=\{1\}$ as $E(G)=\{ 1\}$.

Thus the assumptions of Theorem~\ref{thm:Ols-1} are satisfied, and we can apply it to find a non-elementary hyperbolic group $Q$ with trivial finite radical and an epimorphism $\varphi\colon G \to Q$ which enjoy the desired properties.
\end{proof}

\begin{cor}\label{dense}
Suppose that $P$ is a group-theoretic property such that every group from $\mathcal H_0$ has a quotient $Q\in P(\overline{\mathcal H}_0)$. Then $P(\overline{\mathcal H}_0)$ is dense in $\overline{\mathcal H}_0$.
\end{cor}

\begin{proof}
It suffices to show that for every $(G,X)\in \mathcal H_0$, every neighborhood of $(G,X)$ contains an element of $P(\overline{\mathcal H}_0)$. By the definition of the topology on $\G$, this is equivalent to the following statement: for every $r\in \mathbb N$, there exists a group $(R,Z)\in P(\overline{\mathcal H}_0)$ that is $r$-locally isomorphic to $(G,X)$.

Fix some $r\in \mathbb N$. By applying Lemma \ref{nc} to $G$ and the finite sets $M=B_{G,X}(r)$ and $L=B_{G,X}(2r)\setminus \{1\}$, we can find a group $(Q,Y)\in \mathcal H_0$ and an epimorphism $\varphi \colon G\to Q$ mapping $X$ to $Y$ such that $(Q,Y)$ is $r$-isomorphic to $(G, X)$ and satisfies
\begin{equation}\label{gQ}
\ll \varphi(g) \rr =Q\;\;\;\;\; \forall \, g\in B_{G,X}(2r)\setminus\{ 1\}.
\end{equation}
By our assumption, $Q$ has a quotient group $R\in P(\overline{\mathcal H}_0)$. Let $Z$ be the image of $Y$ in $R$, so that $R=\langle Z \rangle$. The composition of the epimorphisms $G\stackrel{\varphi}\to Q\to R$ must be injective on $B_{G,X}(r)$; indeed, otherwise it sends a non-trivial element $g\in G$ of length $|g|_X\le 2r$ to $1$, which implies $R=\{ 1\}\notin \overline{\mathcal H}_0$ by (\ref{gQ}). It follows that $(G,X)$ and $(R,Z)$ are $r$-locally isomorphic.
\end{proof}

We are now ready to prove the main result of this section.

\begin{proof}[Proof of Proposition \ref{P1P3}]
Since any countable conjunction of generic properties is again generic, it suffices to prove genericity of the individual properties $Tor$, $S$, $VC$, and $QHyp$.

Note that for every topological space $\mathcal X$, any subspace $\S\subseteq \mathcal X$, and any $G_\delta$-subset $\mathcal D \subseteq X$, the intersection $\mathcal D\cap \S$ is a $G_\delta $-subset of $\S$. Hence, the subsets $Tor(\overline{\mathcal H}_0)$, $S(\overline{\mathcal H}_0)$, $VC(\overline{\mathcal H}_0)$,  and $QHyp(\overline{\mathcal H})$  are $G_\delta$-subsets of $\overline{\mathcal H}_0$ by Corollary~\ref{Gdcor}. To complete the proof, it suffices to show that these subsets are dense in $\overline{\mathcal H}_0$.

In \cite[Corollary 2]{Min}, the first author proved that there exists an infinite simple group $Q$, which is a common quotient of all non-elementary hyperbolic groups. This simple group is constructed as the direct limit of a sequence of hyperbolic groups and epimorphisms
\begin{equation}\label{ind}
G_1\stackrel{\e_1}\longrightarrow G_2\stackrel{\e_1}\longrightarrow \ldots
\end{equation}
such that each of the groups $G_i$ is non-elementary and has trivial finite radical; although these properties of $G_i$ are not included in the statement of \cite[Corollary 2]{Min}, they are explicitly stated (and used as an inductive assumption) in the proof. Thus, $Q\in \overline{\mathcal H}_0$ (see Example \ref{limex}). Now we can apply Corollary \ref{dense} and conclude that $S(\overline{\mathcal H}_0)$ and $QHyp(\overline{\mathcal H}_0)$ are dense in $\overline{\mathcal H}_0$.

The same argument works for $Tor(\overline{\mathcal H}_0)$ and $VC(\overline{\mathcal H}_0)$.
The only difference is that we have to refer to \cite[Theorem 2]{MO}.
In this paper, Mikhajlovskii and Olshanskii showed that each non-elementary hyperbolic group $G$ has a non-trivial  torsion  verbally complete quotient. The desired quotient group $L$ is constructed as the direct limit of a sequence of non-elementary hyperbolic groups and epimorphisms (\ref{ind}). In the latter paper, the property $E(G_i)=\{ 1\}$ is explicitly stated (and used) in the proof. Now we can apply Corollary \ref{dense} and conclude that $Tor(\overline{\mathcal H}_0)$ and $VC(\overline{\mathcal H}_0)$ are dense in $\overline{\mathcal H}_0$ as above.
\end{proof}

We conclude this section by mentioning two other applications of our method. Our main emphasis here is on illustrating our approach rather than proving particular results. For this reason, we keep our discussion very brief. Verifying details is left to the readers interested in breeding  quasi-isometrically diverse `monsters'.

Recall that a  that a group $G$ is a \emph{torsion-free Tarski Monster} if $G$ is non-cyclic, finitely generated and every non-trivial proper subgroup of $G$ is infinite cyclic. First examples of such groups were constructed by Olshanskii in \cite{Ols79}. Torsion-free Tarski Monsters with the additional properties listed in the corollary below were constructed by the first author in \cite{Min}.

\begin{cor}\label{Tarski}
There exists $2^{\aleph_0}$ pairwise non-quasi-isometric torsion-free Tarski Monsters that are common quotients of all non-cyclic torsion-free hyperbolic groups and have the property that all maximal subgroups are malnormal and conjugate to each other.
\end{cor}

\begin{proof} Consider the set $\mathcal K$ of all marked groups $(H, \{ a, b\})\in \mathcal G$ such that $H$ is non-cyclic, torsion-free, hyperbolic and $\langle a\rangle $ is malnormal in $H$.
Let $\overline{\mathcal K}$ denote the closure of $\mathcal K$ in $\mathcal{G}$, and let $\mathcal T$ be the set of all marked groups $(G,\{a,b\}) \in \overline{\mathcal K}$ such that:
\begin{enumerate}
  \item[(a)] every element of $G$ is conjugate to a power of $a$;
  \item[(b)] for each $g \in G \setminus \langle a \rangle$ and all $k \in \N$, one has $\langle a^k,g \rangle =G$;
  \item[(c)] $G$ is a quotient of every non-cyclic torsion-free hyperbolic group.
\end{enumerate}
Arguing as in the proof of Proposition~\ref{P1P3} (and using the language with constants $\L^\prime$, introduced in the proof of Proposition \ref{Gd}), it is not hard to show that $\mathcal T$ is a $G_\delta$-set in $\G$. The proof of \cite[Corollary 3]{Min} shows that every group $(H, \{ a, b\})\in \mathcal K$ has a quotient in $\mathcal T$. Arguing as in the proof of Corollary~\ref{dense}, we can show that $\mathcal T$ is dense in $\overline{\mathcal K}$ (one has to use parts (d) and (e) of Theorem \ref{thm:Ols-1} here).  Thus, $\mathcal T$ is comeagre in $\overline{\mathcal K}$. Clearly, every group in $\mathcal T$ is a torsion-free Tarski Monster such that all maximal subgroups are cyclic, malnormal and conjugate to each other.
Both $\mathcal K$ and $\mathcal{T}$ are dense in $\overline{\mathcal K}$ and $\mathcal K \cap \mathcal{T}=\emptyset$ (as no group $G \in \mathcal K$
can satisfy property (a) above), so $\overline{\mathcal K}$ is perfect and we can apply Corollary~\ref{fp}  to get the desired result.
\end{proof}

Finally, we outline a construction of $2^{\aleph_0}$ quasi-isometry classes of finitely generated groups with exactly two conjugacy classes (i.e., groups where all non-trivial elements are conjugate). The first examples of such groups, other than $\Z_2$, were constructed in \cite{Osi10}.

Let $\mathcal {A}$ denote the set of all $(G,X)\in \mathcal G$, where $G$ is a torsion-free finitely presented acylindrically hyperbolic group (see \cite{Osi16} for the definition). Using ideas from \cite{Osi10} and \cite{Hull}, it is possible to show that the set of groups with two conjugacy classes is dense in $\overline{\mathcal A}$. Note also that the property ``all non-trivial elements are conjugate" can be expressed by the $\forall\exists$-sentence
$$
\forall \, g \; \forall\, h\; \exists\, t\;\; (g=1\, \vee \, h=1\, \vee \,  t^{-1}gt=h).
$$
Therefore, the property of having two conjugacy classes is generic in $\overline{\mathcal A}$ and the desired result follows. Note that we cannot consider any set of hyperbolic groups instead of $\mathcal A$ here since groups with two conjugacy classes do not occur among limits of hyperbolic groups in $\G$, as explained in the introduction of \cite{Osi10}.

\section{Quasi-isometric diversity of central extensions}\label{sec:central_ext}
The main goal of this section is to prove Corollary \ref{cbm} but we work in slightly more general settings than necessary.

Let $\{E_i\mid i\in \mathbb N\}$
be a collection of non-trivial abelian groups and let
\begin{equation}\label{E}
E=\bigoplus\limits_{i\in \mathbb N} E_i.
\end{equation}
Let $(G,X) \in \mathcal{G}$ and assume that $G$ contains $E$ as a central subgroup. For each subset $I\subseteq \mathbb N$, we let
\begin{equation}\label{GI}
G_I=G/E_I, \;\;\; {\rm where}\;\;\; E_I=\langle E_i \mid i\in I\rangle .
\end{equation}
Finally, let
\begin{equation}\label{SGI}
\mathcal{S}=\{ (G_I, X_I)\mid I\subseteq \mathbb N\},
\end{equation}
where $X_I$ is the natural image of $X$ in $G_I$.

Let $\mathcal{C}$ denote the Cantor set, which we identify with $2^{\mathbb N}$ endowed with the product topology. In this notation, we have the following.

\begin{lem}\label{S}
The map $\mathcal{C}\to \mathcal G$, given by $I\mapsto (G_I,X_I)$, is continuous. In particular, the set $\mathcal{S}$ is closed in $\mathcal G$.
\end{lem}

\begin{proof}
We think of our map $\mathcal{C}\to \mathcal G$ as a composition of the maps
\[
\mathcal{C}\xrightarrow{\,\;\alpha\,\;} \mathcal N(G)\xrightarrow{\,\;\beta\;\,}  \mathcal G,
\]
where $\alpha (I)= E_I$ for all $I\in \mathcal{C}$ and $\beta $ is the map described in Lemma \ref{NG}.
First let us check that $\alpha$ is continuous. Indeed, for any finite subset $P \subseteq G$ there is a finite subset $K=K(P) \subset \N$ such that
$E \cap P=E_{K} \cap P$. So,  if $I,J \subseteq \N$ satisfy $I \cap K=J \cap K$,
then \[E_I \cap P=E_I \cap E_K \cap P=E_{I\cap K} \cap P=E_{J \cap K} \cap P=E_J \cap E_K \cap P=E_J \cap P.\]
This shows that $\alpha$ is continuous.
By Lemma \ref{NG}, $\beta $ is also continuous, hence so is the composition $\beta \circ \alpha$.
\end{proof}

Thus, we can apply Theorem \ref{main} to $\mathcal{S}$ whenever we can ensure the existence of non-quasi-isometric groups in every non-empty open subset of $\mathcal{S}$.

\begin{prop}\label{ce}
In the notation introduced above, suppose that every $E_i$ is finite. Then the set $\mathcal{S}$ contains either $1$ or $2^{\aleph_0}$ quasi-isometry classes.
\end{prop}

\begin{proof}
Suppose that there exist $I,J\subseteq \mathbb N$ such that $G_I$ and $G_J$ are not quasi-isometric. Let
\[
\mathcal{S}_I=\{ (G_K, X_K)\mid K \subseteq \N \mbox{ and } |K\vartriangle I|<\infty\} \subseteq \mathcal{S}, \]
and similarly we define $\mathcal{S}_J \subseteq \mathcal{S}$.
Note that both $\mathcal{S}_I$ and $\mathcal{S}_J$ are dense in $\mathcal{S}$ and all groups in $\mathcal{S}_I$ (respectively, $\mathcal{S}_J$) are quasi-isometric to $G_I$ (respectively, $G_J$), because they are commensurable up to finite kernels.
This allows us to apply Theorem~\ref{main} and conclude that $\mathcal{S}$ contains $2^{\aleph_0}$ quasi-isometry classes.
\end{proof}

\begin{proof}[Proof of Corollary \ref{cbm}]
In the proof  of \cite[Theorem 7]{Hall} (in the particular case $p=2$), Hall constructed a group $G$, generated by two elements, which splits as
\[
\{1\}\to E\to G\to \mathbb Z_2 {\,\rm wr\,}\mathbb Z \to \{1\},
\]
where the subgroup $E\cong {\Z_2}^\infty $ is central in $G$. We define $G_I$ as above (with $E_i\cong \mathbb Z_2$).

Eskin, Fisher and Whyte \cite[Theorem 1.3]{EFW} showed that if a finitely generated group is quasi-isometric to the lamplighter group $\mathbb Z_2 \,{\rm wr}\, \mathbb Z$, then it has a finite index subgroup that acts properly and cocompactly on the Diestel--Leader graph $\DL(n, n)$ for some $n\in \mathbb N$.

An algebraic description of lattices in the isometry group of the Diestel--Leader graph $\DL(n, n)$ was given by Cornulier, Fisher and Kashyap in \cite{CFK}. In particular, they showed in \cite[Theorem 1.1]{CFK} that every such lattice $R$ contains a normal subgroup $H$ such that $R/H\cong\mathbb Z$; moreover, if $t$ is any element of $R$ mapping onto a generator of $R/H$, then $H$ contains two subgroups $L$ and $L^\prime$ with the following properties:
\begin{equation}\label{L1}
\bigcup_{k\in \Z} t^{-k}Lt^k=\bigcup_{k\in \Z} t^{k}L^\prime t^{-k} =H
\end{equation} and
\begin{equation}\label{L2}
|L\cap L^\prime |<\infty.
\end{equation}

Assume now that $G_\mathbb N\sim _{q.i.} G_\emptyset= \mathbb Z_2 \,{\rm wr}\, \mathbb Z$. Then the lattice $R$ in the group $\operatorname{Isom}(\DL(n,n))$ corresponding to $G_\mathbb N$ will have an infinite center $Z$. In the notation introduced in the previous paragraph, condition (\ref{L1}) obviously implies that $Z\le L\cap L^\prime$, which contradicts (\ref{L2}). Thus $G_\mathbb N\not\sim _{q.i.} G_\emptyset$. Therefore, $\mathcal{S}$ contains $2^{\aleph_0}$ quasi-isometry classes by Proposition~\ref{ce}.
\end{proof}

We also record a general result, which accounts for the Cornulier--Tessera construction of $2^{\aleph_0}$ quasi-isometry classes of solvable groups (see Example \ref{ex2}).

\begin{cor}\label{fpZ}
Let $G$ be a finitely presented group containing a central subgroup of the form $E=\oplus_{i\in \mathbb N} E_i$, where $E_i$ are arbitrary non-trivial abelian groups. Then $G$ has $2^{\aleph_0}$ pairwise non-quasi-isometric quotients.
\end{cor}

\begin{proof}
Let $\mathcal S$ be defined by (\ref{SGI}) and let $\mathcal{S}_0=\{ (G_I, X_I)\mid |I|<\infty\}$. Without loss of generality we  can assume that each $E_i$ is finitely generated. Since $G$ is finitely presented, so is every group from $\mathcal{S}_0$. As $\mathcal{S}_0$ is dense in $\mathcal{S}$ and $\mathcal S$ is perfect by Lemma~\ref{S}, Corollary~\ref{fp} applies.
\end{proof}

Finally, we provide an example of a non-empty perfect subset of $\G$ consisting of quasi-isometric groups. It shows that in Corollary~\ref{fp}
the assumption that $\S$ contains a dense subset of finitely presented groups is essential.

\begin{ex}\label{pqi}
Let $G$ be any $n$-generated group containing a central subgroup of the form $E=\bigoplus_{i\in \mathbb N\cup\{0\}} E_i$, where $E_i\cong \mathbb Z_2$ (e.g., Hall's group used in the proof of Corollary \ref{cbm}). Let $e_i$ be the non-trivial element of $E_i$, $i \in \N$. For every $I\subseteq \N$, we define $Q_I$ to be the quotient of $G$ obtained by imposing relations
$$
e_i=1~~\forall\, i\in I \;\;\; {\rm and}\;\;\; e_j=e_0~~ \forall\, j\in \mathbb N\setminus I.
$$
As above, it is not difficult to check that the groups $Q_I$ (with the appropriate generating sets) form a perfect subset $\mathcal{U}$ of $\mathcal G$. Obviously, all groups from $\mathcal{U}$ are commensurable up to finite kernels  and, therefore, quasi-isometric.
\end{ex}


\begin{thebibliography}{99}
\bibitem[ABJLMS]{ABJLMS}
G. Arzhantseva, M. Bridson, T. Januszkiewicz, I. Leary, A. Minasyan, J. Swiatkowski, Infinite groups with fixed point properties. \emph{Geom. Topol.} \textbf{13} (2009), no. 3, 1229-1263.
%

\bibitem[Bow]{Bow}
 B. Bowditch, Continuously many quasi-isometry classes of $2$-generator groups. \emph{Comment. Math. Helv.} \textbf{73} (1998), no. 2, 232-236.

\bibitem[BC]{BC}
A. Boudec, Y. Cornulier, Densely related groups.  \emph{Ann. Fac. Sci. Toulouse Math. (6)} \textbf{28} (2019), no. 4, 619-653.

\bibitem[BZ]{BZ}
J. Brieussel, T. Zheng, Speed of random walks, isoperimetry and compression of finitely generated groups. \emph{Ann. of Math. (2)} \textbf{193} (2021), no. 1, 1-105.

\bibitem[BH]{BH} M.R. Bridson, A. Haefliger,
Metric spaces of non-positive curvature.
Grundlehren der Mathematischen Wissenschaften [Fundamental Principles of Mathematical Sciences] \textbf{319}. \emph{Springer-Verlag, Berlin}, 1999.

\bibitem[Cha]{Cha}
C. Champetier, L'espace des groupes de type fini. \emph{Topology} \textbf{39} (2000), no. 4, 657-680.

\bibitem[CFK]{CFK}
Y. Cornulier, D. Fisher, N. Kashyap, Cross-wired lamplighter groups.
\emph{New York J. Math.} \textbf{18} (2012), 667-677.

\bibitem[CT]{CT}
Y. Cornulier, R. Tessera, Dehn function and asymptotic cones of Abels' group.  \emph{J. Topol.} \textbf{6} (2013), no. 4, 982-1008.

\bibitem[DM]{DM}
C. Dru{t}u, J.M. Mackay, Random groups, random graphs and eigenvalues of $p$-Laplacians, \emph{Adv. Math.} \textbf{341} (2019), 188-254.


\bibitem[EFW]{EFW}
A. Eskin, D. Fisher, K. Whyte, Coarse differentiation of quasi-isometries II: Rigidity for Sol and lamplighter groups. \emph{Ann. of Math.} \textbf{177} (2013), no. 3, 869-910.

\bibitem[FW]{FW}
K. Fujiwara, K. Whyte, A note on spaces of asymptotic dimension one. \emph{Algebr. Geom. Topol.} \textbf{7} (2007), 1063-1070.

\bibitem[GS]{GS}
D. Gruber, A. Sisto, Divergence and quasi-isometry classes of random Gromov's monsters. \emph{Math Proc. Cam. Phil. Soc.}, to appear. \texttt{arXiv:1805.04039}.

\bibitem[Gri]{Gri}
R. Grigorchuk, Degrees of growth of finitely generated groups and the theory of invariant means. (Russian) \emph{Izv. Akad. Nauk SSSR}, Ser. Mat. \textbf{48} (1984), no. 5, 939-985.

\bibitem[Gro]{Gro}
M. Gromov, Hyperbolic groups. \emph{Essays in group theory}, 75-263, Math. Sci. Res. Inst. Publ., 8, \emph{Springer, New York}, 1987.

\bibitem[Gub]{Guba}
V. Guba, A finitely generated complete group. (Russian) \emph{Izv. Akad. Nauk SSSR Ser. Mat.} \textbf{50} (1986), no. 5, 883-924.

\bibitem[Hal]{Hall}
P. Hall, Finiteness conditions for soluble groups. \emph{Proc. London Math. Soc.} \textbf{4} (1954), no. 3, 419-436.

\bibitem[Hul]{Hull}
M. Hull, Small cancellation in acylindrically hyperbolic groups. \emph{Groups Geom. Dyn.} \textbf{10} (2016), no. 4, 1077-1119.


\bibitem[Kec]{Kech}
A. Kechris, {Classical descriptive set theory}. Graduate Texts in Mathematics \textbf{156}. \emph{Springer-Verlag, New York}, 1995.

\bibitem[KOO]{KOO}
A.A. Klyachko, A.Yu. Olshanskii, D.V. Osin, On topologizable and non-topologizable groups.  \emph{Topology Appl.} \textbf{160} (2013), no. 16, 2104-2120.

\bibitem[KLS]{KLS}
R. Kropholler, I.J. Leary, I. Soroko, Uncountably many quasi-isometry classes of groups of type $\mathit{FP}$.  \emph{Amer. J. Math.} \textbf{142} (2020), no. 6, 1931-1944.


\bibitem[LS]{LS} R.C. Lyndon, P.E. Schupp, Combinatorial group theory.
Ergebnisse der Mathematik und ihrer Grenzgebiete, Band 89. \emph{Springer-Verlag, Berlin-New York,} 1977.

\bibitem[Mar]{Mar16}
D. Marker, Lectures on Infinitary Model Theory. Lecture Notes in Logic \textbf{46}. \emph{Cambridge University Press, Cambridge}, 2016.


\bibitem[MO]{MO}
K. Mikhajlovskii, A. Olshanskii, Some constructions relating to hyperbolic groups. \emph{Geometry and cohomology in group theory (Durham, 1994)}, 263-290,
London Math. Soc. Lecture Note Ser. \textbf{252}, \emph{Cambridge Univ. Press, Cambridge}, 1998.

\bibitem[Min]{Min} A. Minasyan, On residualizing homomorphisms preserving quasiconvexity. \emph{Comm. Algebra} \textbf{33} (2005), no. 7, 2423-2463.

\bibitem[Myc]{Myc}
J. Mycielski, Independent sets in topological algebras. \emph{Fund. Math.} \textbf{55} (1964), no. 2, 139-147.

\bibitem[NS]{NS}
A. Naor and L. Silberman, Poincar\'e inequalities, embeddings, and wild groups, \emph{Compos. Math.} \textbf{147} (2011). no. 5, 1546-1572.


\bibitem[Ols79]{Ols79} A.Yu. Olshanskii,
An infinite simple torsion-free Noetherian group. (Russian)
\emph{Izv. Akad. Nauk SSSR Ser. Mat.} \textbf{43} (1979), no. 6, 1328-1393.

\bibitem[Ols93]{Ols}
A.Yu. Olshanskii, On residualing homomorphisms and G-subgroups of hyperbolic groups. \emph{Internat. J. Algebra Comput.} \textbf{3} (1993), no. 4, 365-409.

\bibitem[Ols80]{Ols80}
A.Yu. Olshanskii, On the question of the existence of an invariant mean on a group. (Russian)
\emph{Uspekhi Mat. Nauk} \textbf{35} (1980), no. 4(214), 199-200.

\bibitem[Ols82]{Ols82}
A.Yu. Olshanskii, The Novikov-Adyan theorem. \emph{Mat. Sb.} \textbf{118} (1982), no. 2, 203-235.


\bibitem[Osi02]{Osi}
D. Osin, Weakly amenable groups. \textit{Contemp.
Math.} {\bf 298} (2002), 105-113.

\bibitem[Osi10]{Osi10}
D. Osin, Small cancellations over relatively hyperbolic groups and embedding theorems. \emph{Ann. Math.} {\bf 172} (2010), no. 1, 1-39.

\bibitem[Osi16]{Osi16} D. Osin, Acylindrically hyperbolic groups. \emph{Trans. Amer. Math. Soc.} \textbf{368} (2016), no. 2, 851-888.

\bibitem[Osi21]{Osi21} D. Osin, A topological zero-one law and elementary equivalence of finitely generated groups.  \emph{Ann. Pure Appl. Logic} \textbf{172} (2021), no. 3, 102915.

\bibitem[Rap]{Raph} D. Raphael, Commensurability and elementary equivalence of polycyclic groups.
\emph{Bull. Austral. Math. Soc.} \textbf{53} (1996), no. 3, 425-439.

\bibitem[Tho]{T}
S. Thomas, On the complexity of the quasi-isometry and virtual isomorphism problems for finitely generated groups. \emph{Groups Geom. Dyn.} \textbf{2} (2008), no. 2, 281-307.

\end{thebibliography}
\end{document}